\documentclass{amsart}
\usepackage{graphicx}
\usepackage{amssymb}

\newtheorem{thm}{Theorem}[section]

\newtheorem{lem}[thm]{Lemma}

\theoremstyle{definition}

\theoremstyle{remark}
\newtheorem{rem}[thm]{Remark}

\begin{document}

\title[Entire solutions to equations of minimal surface type]{Entire solutions to equations of minimal surface type in six dimensions}
\author{Connor Mooney}
\address{Department of Mathematics, UC Irvine}
\email{\tt mooneycr@math.uci.edu}

\begin{abstract}
We construct nonlinear entire solutions in $\mathbb{R}^6$ to equations of minimal surface type that correspond to parametric elliptic functionals. 
\end{abstract}
\maketitle

\section{Introduction}
A well-known theorem of Bernstein says that entire minimal graphs in $\mathbb{R}^3$ are hyperplanes. Building on work of Fleming \cite{F} and De Giorgi \cite{DeG}, Simons \cite{S} extended this result to minimal graphs in $\mathbb{R}^{n+1}$ for $n \leq 7$. In contrast, there are nonlinear entire solutions to the minimal surface equation in dimension $n \geq 8$ due to Bombieri-De Giorgi-Giusti \cite{BDG} and Simon \cite{Si1}.

In this paper we study the Bernstein problem for a more general class of parametric elliptic functionals. These assign to an oriented hypersurface $\Sigma \subset \mathbb{R}^{n+1}$ the value
\begin{equation}\label{PhiEnergy}
A_{\Phi}(\Sigma) := \int_{\Sigma} \Phi(\nu),
\end{equation}
where $\nu$ is a choice of unit normal to $\Sigma$ and $\Phi \in C^{2,\,\alpha}(\mathbb{S}^n)$ is a positive even function. We say $\Phi$ is uniformly elliptic if its one-homogeneous extension to $\mathbb{R}^{n+1}$ has uniformly convex level sets. The case $\Phi = 1$ corresponds to the area functional. In the general case, the minimizers of $A_{\Phi}$ model crystal surfaces (see \cite{M} and the references therein). Below we assume $\Phi$ is uniformly elliptic unless otherwise specified.

When a critical point of $A_{\Phi}$ can be written as the graph of a function $u$ on a domain $\Omega \subset \mathbb{R}^n$, we say that $u$ is $\Phi$-minimal. It solves an elliptic equation of minimal surface type (see Section \ref{Preliminaries}). Jenkins \cite{J} proved that global $\Phi$-minimal functions are linear in dimension $n = 2$. Simon \cite{Si2} extended this result to dimension $n = 3$, using an important regularity theorem of Almgren-Schoen-Simon \cite{ASS} for minimizers of the parametric problem. He also showed that the result holds up to dimension $n = 7$ when $\Phi$ is close in an appropriate sense to the area integrand.

The purpose of this paper is to construct a nonlinear entire $\Phi$-minimal function on $\mathbb{R}^6$, for an appropriate uniformly elliptic integrand (which is necessarily far from the area integrand). Our main theorem is:

\begin{thm}\label{Main}
There exists a quadratic polynomial $u$ on $\mathbb{R}^6$ that is $\Phi$-minimal for a uniformly elliptic integrand $\Phi \in C^{2,\,1}(\mathbb{S}^6)$.
\end{thm}

\noindent Theorem \ref{Main} settles the Bernstein problem for equations of minimal surface type in dimension $n \geq 6$, leaving open the cases $n = 4,\,5$. It also answers the question whether or not there exists a nonlinear polynomial that solves such an equation. It remains an interesting open question whether or not there exists a nonlinear polynomial that solves the minimal surface equation.
 
Our approach to constructing entire solutions is different from the one taken by Bombieri-De Giorgi-Giusti,
which is based on constructing super- and sub- solutions. We instead fix $u$, which reduces the problem to solving a linear hyperbolic equation for $\Phi$. It turns out that in $\mathbb{R}^6$, we can choose a quadratic polynomial $u$ such that the solutions to this hyperbolic equation are given by an explicit representation formula. By prescribing the Cauchy data carefully we obtain an integrand with the desired properties.

As a consequence of Theorem \ref{Main} we show that the cone over $\mathbb{S}^2 \times \mathbb{S}^2$ in $\mathbb{R}^6$ minimizes $A_{\Phi_0}$, where $\Phi_0$ is the restriction of the integrand $\Phi$ from Theorem \ref{Main} to $\mathbb{S}^6 \cap \{x_7 = 0\}$ (see Remark \ref{Cones}). In fact, each level set of the function $u$ from Theorem \ref{Main} minimizes $A_{\Phi_0}$. (This observation is what guided us to the example). Morgan \cite{M} previously showed that the cone over $\mathbb{S}^{k} \times \mathbb{S}^k$ in $\mathbb{R}^{2k+2}$ minimizes a parametric elliptic functional for each $k \geq 1$, using the method of calibrations.

The analogue of the quadratic polynomial $u$ from Theorem \ref{Main} in dimension $n = 4$ is not $\Phi$-minimal for any uniformly elliptic integrand $\Phi$ that enjoys certain natural symmetries (see Remark \ref{No4D}). However, it is feasible that our approach could produce entire $\Phi$-minimal functions in the lowest possible dimension $n = 4$, that have sub-quadratic growth (see Remark \ref{Generalizations}).

\section*{Acknowledgements}
The author is grateful to Richard Schoen, Brian White, and Yu Yuan for inspiring discussions on topics related to this research.
The research was supported by NSF grant  DMS-1854788.

\section{Preliminaries}\label{Preliminaries}

\subsection{Legendre Transform}
Let $w$ be a smooth function on a domain $\Omega \subset \mathbb{R}^n$, and assume that $\nabla w$ is a diffeomorphism with inverse $X$. We define the Legendre transform $w^*$ on the image of $\nabla w$ by
$$w^*(p) := p \cdot X(p) - w(X(p)).$$
Differentiating two times we obtain
\begin{equation}\label{LTDeriv}
\nabla w^*(p) = X(p), \quad D^2w^*(p) = (D^2w)^{-1}(X(p)).
\end{equation}

\subsection{Euler-Lagrange Equation}
Assume that $\Phi \in C^{2,\,\alpha}(\mathbb{S}^n)$ is a positive, uniformly elliptic integrand. Here and below we will identify $\Phi$ with its one-homogeneous extension to $\mathbb{R}^{n+1}$, and uniform ellipticity
means that $\{\Phi = 1\}$ is uniformly convex. 

If $\Sigma$ is the graph of a smooth function $u$ on a domain $\Omega \subset \mathbb{R}^n$ then we can rewrite the variational integral (\ref{PhiEnergy}) as
$$A_{\Phi}(\Sigma) = \int_{\Omega} \varphi(\nabla u)\,dx,$$
where 
\begin{equation}\label{phiDef}
\varphi(p) := \Phi(-p,\,1).
\end{equation}
Thus, if $\Sigma$ is a critical point of $A_{\Phi}$ then $u$ solves the Euler-Lagrange equation
\begin{equation}\label{UsualEL}
\text{div}(\nabla \varphi(\nabla u)) = \varphi_{ij}(\nabla u)u_{ij} = 0
\end{equation}
in $\Omega$. The function $\varphi$ is locally uniformly convex (by the uniform ellipticity of $\Phi$), but the ratio of the minimum to maximum eigenvalues of $D^2\varphi$ degenerates at infinity.
Thus the equation (\ref{UsualEL}) is a quasilinear degenerate elliptic PDE for $u$, known in the literature as a variational equation of minimal surface or mean curvature type (see e.g. Chapter $16$ in \cite{GT} and
the references therein).

Our approach is to rewrite (\ref{UsualEL}) as a linear equation for $\varphi$. Assume that $\nabla u$ is a smooth diffeomorphism. Then using the relations in (\ref{LTDeriv}) we can rewrite the equation
(\ref{UsualEL}) as
\begin{equation}\label{LegEL}
(u^*)^{ij}(y)\varphi_{ij}(y) = 0
\end{equation}
for $y$ in the image of $\nabla u$. Below we will fix $u^*$, and then solve the equation (\ref{LegEL}) for $\varphi$.

\begin{rem}
In parametric form, the Euler-Lagrange equation (\ref{UsualEL}) for a critical point $\Sigma$ of $A_{\Phi}$ is
\begin{equation}\label{ParametricEL}
\text{tr}(D^2\Phi(\nu^{\Sigma}(x)) \cdot II^{\Sigma}(x)) = \Phi_{ij}(\nu^{\Sigma}(x))II^{\Sigma}_{ij}(x) = 0,
\end{equation}
where $\nu^{\Sigma}$ is the Gauss map of $\Sigma$ and $II^{\Sigma}$ is the second fundamental form of $\Sigma$. We note that (\ref{ParametricEL}) is invariant under dilations of $\Sigma$. Equation
(\ref{UsualEL}) can be viewed as the projection of the equation (\ref{ParametricEL}) onto a hyperplane.
\end{rem}

\begin{rem}
The graph $\Sigma$ of an entire solution to (\ref{UsualEL}) is not only a critical point, but a minimizer of $A_{\Phi}$. One way to see this is to observe that the translations of $\Sigma$ in the $x_{n+1}$ direction foliate either
side of $\Sigma$. Another way is to extend the unit normal $\nu$ on $\Sigma$ to $\mathbb{R}^{n+1}$ by letting it be constant in the $x_{n+1}$ direction, and then show that $\nabla\Phi(\nu)$ is a calibration. Indeed,
$\nabla \Phi(\nu)$ is divergence-free in $\mathbb{R}^{n+1}$ by the equation (\ref{ParametricEL}), and by viewing $\Phi$ as the support function of the uniformly convex hypersurface $K := \nabla \Phi(\mathbb{S}^{n})$ we see that
$$\nabla\Phi(\nu) \cdot \tilde{\nu} \leq \Phi(\tilde{\nu})$$ 
for any $\nu,\,\tilde{\nu} \in \mathbb{S}^n$, with equality if and only if $\nu = \tilde{\nu}$.
\end{rem}

\section{Proof of Theorem \ref{Main}}\label{Proof}
We denote points in $\mathbb{R}^6$ by $(p,\,q)$, with $p,\,q \in \mathbb{R}^{3}$. The polynomial $u$ from Theorem \ref{Main} is
\begin{equation}\label{uDef}
u(p,\,q) := \frac{1}{2}(|p|^2 - |q|^2).
\end{equation}
We note that $u = u^*$. Below we let $\Box$ denote the wave operator $\partial_x^2 - \partial_y^2$ on $\mathbb{R}^2$.

\begin{lem}\label{Reduction}
To prove Theorem \ref{Main} it suffices to find an analytic function $\psi(x,\,y)$ on $\mathbb{R}^2$ that is even in $x$ and $y$, solves the PDE
\begin{equation}\label{2DEL}
\Box \psi + 2\,\nabla \psi \cdot \left(\frac{1}{x},\,-\frac{1}{y}\right) = 0
\end{equation}
in the positive quadrant, and satisfies that the one-homogeneous function
$$\Psi(x,\,y,\,z) = |z|\,\psi\left(\frac{x}{z},\,\frac{y}{z}\right)$$
on $\mathbb{R}^3 \backslash \{z = 0\}$ has a continuous extension to $\mathbb{R}^3$ that is positive and locally $C^{2,\,1}$ on $\mathbb{R}^3 \backslash \{0\}$, and has uniformly convex level sets.
\end{lem}
\begin{proof}
Suppose we have found such a function $\psi$, and denote points in $\mathbb{R}^7$ by $(p,\,q,\,z)$ with $p,\,q \in \mathbb{R}^3$ and $z \in \mathbb{R}$.  Then the function
$$\Phi(p,\,q,\,z) := \Psi(|p|,\,|q|,\,z)$$
satisfies the desired regularity and convexity conditions. Furthermore, if we define $\varphi$ by the relation (\ref{phiDef}), that is,
$$\varphi(p,\,q) := \Phi(-p,\,-q,\,1) = \psi(|p|,\,|q|),$$
then by the definition (\ref{uDef}) of $u$ and the equation (\ref{2DEL}) for $\psi$ we have
$$(u^*)^{ij}\varphi_{ij} = 0$$
on $\mathbb{R}^6$. Hence equation (\ref{LegEL}) holds and the function $u$ is $\Phi$-minimal.
\end{proof}

\begin{proof}[{\bf Proof of Theorem \ref{Main}}]
We note that a function $\psi$ solves (\ref{2DEL}) in the positive quadrant if and only if 
$$\Box(x\,y\,\psi) = 0.$$ 
The general solution to (\ref{2DEL}) is thus given by the formula
$$\psi(x,\,y) = \frac{f(x+y) + g(x-y)}{xy}.$$
We will show that the choice
$$f(s) = - g(s) = 2^{-\frac{5}{2}}(2+s^2)^{3/2}$$
gives a function $\psi$ satisfying the remaining conditions of Lemma \ref{Reduction}. 

After rotating the plane by $\frac{\pi}{4}$ (and for ease of notation continuing to denote the coordinates by $x$ and $y$) we have for the above choices of $f$ and $g$ that
\begin{align*}
\psi(x,\,y) &= \frac{(1+x^2)^{3/2} - (1+y^2)^{3/2}}{x^2-y^2} \\
&= \frac{A^2+AB+B^2}{A+B},
\end{align*}
where
$$A := (1+x^2)^{1/2}, \quad B := (1+y^2)^{1/2}.$$
Hence $\psi$ is positive, analytic, and invariant under reflection over the axes and the diagonals. Furthermore, $\psi$ is locally uniformly convex. Indeed, after some calculation (which we omit) we arrive at
$$\det D^2\psi = 3\,(A+B)^{-4}\left(2 + \frac{1}{AB}\right) > 0,$$
and since 
$$D^2\psi(0,\,0) = \frac{3}{4}I$$ 
we conclude that $D^2\psi$ is everywhere positive definite.

Now let
\begin{align*}
\Psi(x,\,y,\,z) &:= |z|\,\psi\left(\frac{x}{z},\,\frac{y}{z}\right) \\
&= \frac{(x^2 + z^2)^{3/2} - (y^2 + z^2)^{3/2}}{x^2-y^2} \\
&=\frac{D^2 + DE + E^2}{D+E},
\end{align*}
where 
$$D := (x^2+z^2)^{1/2}, \quad E := (y^2 + z^2)^{1/2}.$$
By the local uniform convexity and analyticity of $\psi$ and the one-homogeneity of $\Psi$, we just need to check that $\Psi \in C^{2,\,1}$ in a neighborhood of the circle $\mathbb{S}^2 \cap \{z = 0\}$,
and that on this circle the Hessian of $\Psi$ restricted to any plane tangent to $\mathbb{S}^2$ is positive definite.

Restricting $\Psi$ to the plane $\{x = 1\}$ we get a function of $y$ and $z$ that is $C^{2,\,1}$ in a neighborhood of the origin (and analytic away from the origin), and at $(1,\,0,\,0)$ we have 
$$\Psi_{yy} = 2,\, \Psi_{yz} = 0,\, \Psi_{zz} = 3.$$ 
By the symmetries of $\Psi$ this gives the result in a neighborhood of the points $(\pm 1,\,0,\,0)$ and $(0,\,\pm 1,\,0)$.
We may thus restrict our attention to the region 
$$\Omega_{\delta} := \{|x|,\,|y| \geq \delta\}$$ 
for $\delta > 0$ sufficiently small. In the region $\Omega_{\delta} \cap \left\{|z| < \frac{\delta}{2}\right\}$ the function $\Psi$ is analytic, and has the expansion
\begin{equation}\label{Expansion}
\begin{split}
\Psi(x,\,y,\,z) &= \frac{x^2+|xy|+y^2}{|x|+|y|} + \frac{3}{2}\frac{z^2}{|x|+|y|} \\
&- \frac{1}{|x|+|y|}\, \sum_{k \geq 2} a_k\,\left(\sum_{i=0}^{2k-4} \frac{1}{|x|^{i+1}\,|y|^{2k-3-i}}\right) z^{2k},
\end{split}
\end{equation}
where $a_k$ are the coefficients in the Taylor series of $(1+s)^{3/2}$ around $s = 0$. Thus, for any unit vector $e \in \{z = 0\}$ we have on $\Omega_{\delta} \cap \{z = 0\}$ that 
$$\Psi_{ez} = 0,\, \Psi_{zz} = \frac{3}{|x| + |y|}.$$ 
It only remains to to check that the first term in (\ref{Expansion}) is locally uniformly convex on lines in $\{z = 0\}$ that don't pass through the origin. By its one-homogeneity and symmetry in $x$ and $y$, it suffices to check this on the line $\{x = 1\}$. Since
$$\Psi(1,\,y,\,0) = |y| + \frac{1}{1+|y|}$$
is locally uniformly convex, we are done.
\end{proof}

\begin{rem}\label{IntegrandFormula}
The integrand from Theorem \ref{Main} is given explicitly by the formula
\begin{equation}\label{PhiFormula}
\Phi(p,\,q,\,z) = \frac{\left((|p|+|q|)^2+2z^2\right)^{3/2}-\left((|p|-|q|)^2+2z^2\right)^{3/2}}{2^{5/2}|p||q|},
\end{equation}
where $p,\,q \in \mathbb{R}^3$ and $z \in \mathbb{R}$.
\end{rem}

\begin{rem}\label{Cones}
Theorem \ref{Main} implies that the cone $C$ over $\mathbb{S}^2 \times \mathbb{S}^2$ in $\mathbb{R}^6$ is a minimizer of $A_{\Phi_0}$, where 
$$\Phi_0(p,\,q) = \frac{\bigl\lvert |p| + |q| \bigr\rvert ^3 - \bigl\lvert |p| - |q| \bigr\rvert ^3}{2^{5/2}|p||q|}$$
is the restriction of $\Phi$ (defined by (\ref{PhiFormula})) to the hyperplane $\{z = 0\}$. Indeed, the hypersurfaces $\{u = \pm 1\}$ are critical points of $A_{\Phi_0}$, and their dilations foliate either side of $C$.
To see e.g. that $\{u = 1\}$ is a critical point of $A_{\Phi_0}$, first note that $R\,u$ is $\Phi$-minimal for all $R > 0$ by the homogeneity of $u$ and the invariance of the equation (\ref{UsualEL}) under the rescalings
$u \rightarrow R^{-1}u(Rx)$. Then write the equation (\ref{ParametricEL}) for the graph of $R\,u$ over points in $\{u = 1\}$, and pass to the limit as $R \rightarrow \infty$.
\end{rem}

\begin{rem}\label{No4D}
The analogue of the quadratic polynomial (\ref{uDef}) in $\mathbb{R}^4$, where $p,\,q \in \mathbb{R}^2$, is not $\Phi$-minimal for any uniformly elliptic integrand $\Phi$ on $\mathbb{S}^4$ that depends only on $|p|,\, |q|$ and $z$. To see this, we first observe that by the reasoning in Remark \ref{Cones} it suffices to show that $\{u = 1\}$ is not a critical point of $A_{\Phi_0}$ for any uniformly elliptic integrand $\Phi_0$ on $\mathbb{S}^3$ that depends only on $|p|$ and $|q|$. When we fix $\Sigma := \{u = 1\}$ and impose that $\Phi_0$ depends only on $|p|$ and $|q|$, the equation (\ref{ParametricEL}) reduces to an ODE. By analyzing this ODE one can show that one eigenvalue of $D^2\Phi_0$ will tend to infinity on the Clifford torus $\mathbb{S}^1 \times \mathbb{S}^1$.
\end{rem}

\begin{rem}\label{Generalizations}
If we take $u^*(p,\,q) = \frac{1}{m}(|p|^m - |q|^m)$ and $\varphi(p,\,q) = \psi(|p|,\,|q|)$ with $p,\,q \in \mathbb{R}^{k+1}$, then equation (\ref{LegEL}) is equivalent to the hyperbolic PDE
$$\frac{1}{m-1}x^{2-m}\psi_{xx} + k\,x^{1-m}\psi_x = \frac{1}{m-1}y^{2-m}\psi_{yy} + k\,y^{1-m}\psi_y$$
for $\psi$ in the positive quadrant. The Cauchy problem for this equation can be solved in terms of certain hypergeometric functions (see \cite{C} and the references therein). In special cases the representation formula is particularly simple, e.g. when $k=m=2$ (treated above), or when $k=1$ and $m=4$, in which case the general solution is
$$\psi(x,\,y) = \frac{f(x^2+y^2) + g(x^2-y^2)}{x^2y^2}.$$
The corresponding integrand $\Phi$ (constructed as in the proof of Lemma \ref{Reduction}) is not uniformly elliptic for any choice of $f$ and $g$, because the maximum and minimum principal curvatures of the graph of
$$u = \frac{3}{4}\left(|p|^{\frac{4}{3}} - |q|^{\frac{4}{3}}\right)$$
are not of comparable size near $\{|p|\,|q| = 0\}$. However, it is feasible that for a judicious choice of $f$ and $g$, one could make a small perturbation of the corresponding integrand and then use the method of super- and sub-solutions to construct an entire solution to a variational equation of minimal surface type in $\mathbb{R}^4$ that grows at the same rate as $u$.
\end{rem}




\end{document}